\documentclass[reqno,11pt,english]{amsart}
\usepackage[T1]{fontenc}
\usepackage[latin9]{inputenc}
\usepackage{amsthm}
\usepackage{amstext}
\usepackage{amssymb}
\usepackage{esint}

\makeatletter
\numberwithin{equation}{section}
\numberwithin{figure}{section}
\theoremstyle{plain}
\newtheorem{thm}{\protect\theoremname}[section]
  \theoremstyle{plain}
  \newtheorem{prop}[thm]{\protect\propositionname}
  \theoremstyle{plain}
  \newtheorem{lem}[thm]{\protect\lemmaname}
  \theoremstyle{plain}
  \newtheorem{cor}[thm]{\protect\corollaryname}

\usepackage{amsmath,amsfonts,amssymb,amsthm,epsfig}
\usepackage{verbatim}

\usepackage{color}

\usepackage[final,allcolors=blue,colorlinks=true]{hyperref}
\usepackage[leqno]{amsmath}

\def\R{\mathbb R}

\def\cal{\mathcal}
\def\al{\alpha}
\def\be{\beta}
\def\ga{\gamma}
\def\de{\delta}
\def\ep{\epsilon}
\def\la{\lambda}
\def\si{\sigma}
\def\th{\theta}
\def\var{\varphi}

\def\na{\nabla}

\def\Om{\Omega}
\def\De{\Delta}

\def\pa{\partial}
\def\divergence{{\rm div}\,}

\newcommand\esssup{{\rm \,esssup\,}}

\numberwithin{equation}{section}
\theoremstyle{definition}
\begingroup

\endgroup

\makeatother

\usepackage{babel}
  \providecommand{\corollaryname}{Corollary}
  \providecommand{\lemmaname}{Lemma}
  \providecommand{\propositionname}{Proposition}
\providecommand{\theoremname}{Theorem}

\begin{document}

\title{\title[Gradient Estimates for Quasilinear Equations]GGradient Estimates
for Solutions To Quasilinear Elliptic Equations with Critical Sobolev
Growth and Hardy Potential}

\author{Chang-Lin Xiang}

\date{\today}

\address{Department of Mathematics and Statistics, P.O. Box 35 (MaD) FI-40014
University of Jyv\"askyl\"a, Finland}

\email{changlin.c.xiang@jyu.fi}
\begin{abstract}
This note is a continuation of the work \cite{CaoXiangYan2014}. We
study the following quasilinear elliptic equations
\[
-\Delta_{p}u-\frac{\mu}{|x|^{p}}|u|^{p-2}u=Q(x)|u|^{\frac{Np}{N-p}-2}u,\quad\, x\in\mathbb{R}^{N},
\]
where $1<p<N,0\leq\mu<\left((N-p)/p\right)^{p}$ and $Q\in L^{\infty}(\R^{N})$.
Optimal asymptotic estimates on the gradient of solutions are obtained
both at the origin and at the infinity.
\end{abstract}

\maketitle
{\small

\tableofcontents{}

\section{Introduction and main result}

Let $1<p<N,0\leq\mu<\bar{\mu}=\left((N-p)/p\right)^{p}$ and $p^{*}=Np/(N-p)$.
In this note, we study the following quasilinear elliptic equations
\begin{equation}
-\Delta_{p}u-\frac{\mu}{|x|^{p}}|u|^{p-2}u=Q(x)|u|^{p^{*}-2}u,\quad\, x\in\mathbb{R}^{N},\label{eq: Objects}
\end{equation}
where
\begin{eqnarray*}
\Delta_{p}u=\sum_{i=1}^{N}\partial_{x_{i}}(|\nabla u|^{p-2}\partial_{x_{i}}u), &  & \nabla u=(\partial_{x_{1}}u,\cdots,\partial_{x_{N}}u),
\end{eqnarray*}
is the $p$-Laplacian operator and $Q\in L^{\infty}(\mathbb{R}^{N})$.

Let $C_{0}^{\infty}(\R^{N})$ be the space of smooth functions in
$\R^{N}$ with compact support and ${\cal D}^{1,p}(\R^{N})$ the closure
of $C_{0}^{\infty}(\R^{N})$ in the seminorm $||v||_{{\cal D}^{1,p}(\R^{N})}=||\na v||_{L^{p}(\R^{N})}$.
A function $u\in{\cal D}^{1,p}(\R^{N})$ is a weak solution to equation
(\ref{eq: Objects}) if
\begin{eqnarray*}
\int_{\mathbb{R}^{N}}\left(|\nabla u|^{p-2}\nabla u\cdot\nabla\varphi-\frac{\mu}{|x|^{p}}|u|^{p-2}u\varphi\right)=\int_{\mathbb{R}^{N}}Q(x)|u|^{p^{*}-2}u\varphi &  & \forall\:\var\in C_{0}^{\infty}(\R^{N}).
\end{eqnarray*}

In \cite{CaoXiangYan2014}, the author obtained the following result
on the asymptotic behaviors of solutions to equation (\ref{eq: Objects})
both at the origin and at the infinity.
\begin{thm}
\label{thm: Cao-Xiang-Yan} Let $Q\in L^{\infty}(\mathbb{R}^{N})$
and $u\in\mathcal{D}^{1,p}(\mathbb{R}^{N})$ be a weak solution to
equation (\ref{eq: Objects}). Then there exists a positive constant
$C$ depending on $N,p,\mu,||Q||_{\infty}$ and the solution $u$
such that
\begin{eqnarray}
|u(x)|\le C|x|^{-\gamma_{1}} &  & for\;|x|<R_{0},\label{eq: origin growth}
\end{eqnarray}
 and that
\begin{eqnarray*}
|u(x)|\leq C|x|^{-\gamma_{2}} &  & for\;|x|>R_{1},
\end{eqnarray*}
where $0<R_{0}<1<R_{1}$ are constants depending on $N,p,\mu,||Q||_{\infty}$
and the solution $u$.
\end{thm}
In the above theorem and in the following, the exponents $\ensuremath{\gamma_{1}}$
and $\ga_{2}$ are defined as follows: consider the equation
\begin{eqnarray*}
(p-1)\ga^{p}-(N-p)\ga^{p-1}+\mu=0, &  & \ga\ge0.
\end{eqnarray*}
Due to our assumptions on $N,p$ and $\mu$, that is, $1<p<N$ and
$0\le\mu<\bar{\mu},$ above equation has two nonnegative solutions
$\ga_{1}$ and $\ga_{2}$ and they satisfy
\[
0\le\ga_{1}<\frac{N-p}{p}<\ga_{2}\le\frac{N-p}{p-1}.
\]

Note that the constants $C,R_{0},R_{1}$ depend on the solution $u$.
This dependence has been discussed in \cite{CaoXiangYan2014} in full
details. Later in this note we will give a brief discussion on this
dependence after giving our main result.

Asymptotic estimates for solutions to equation (\ref{eq: Objects})
and to its variants are  useful. For applications of such estimates,
we refer to e.g. \cite{Cab,Cao,Cao2,Cao4,Cao3,F}. In the present
note, we continue the work of \cite{CaoXiangYan2014} and study asymptotic
behaviors of gradient of weak solutions to equation (\ref{eq: Objects}).
Not much is known in this aspect.

To the best of our knowledge, all known results on the asymptotic
behaviors of gradient of weak solutions to equation (\ref{eq: Objects})
are concerned with the special case in which $Q\equiv1$. Let us discuss
the known results according to the value of the parameter $\mu$.

In the case when $\mu=0$, a prototype of equation (\ref{eq: Objects})
when $Q\equiv$1 is
\begin{eqnarray}
-\De_{p}u=|u|^{p^{*}-2}u, &  & \text{in }\R^{N}.\label{eq: mu is 0}
\end{eqnarray}
When $p=2$, Gidas, Ni and Nirenberg \cite{GNN}   proved that positive
$C^{2}$ solutions of equation (\ref{eq: mu is 0}) (not necessarily
in ${\cal D}^{1,2}(\R^{N})$) satisfying
\begin{equation}
\liminf_{|x|\to\infty}\left(|x|^{N-2}u(x)\right)<\infty\label{eq: GNN condition}
\end{equation}
 must be of the form $u(x)=u_{0}^{\la,x_{0}}(x)=\lambda^{\frac{N-2}{2}}u_{0}(\lambda(x-x_{0}))$
for some $\lambda>0$ and some $x_{0}\in\mathbb{R}^{N}$,  where
\[
u_{0}(x)=(N(N-2))^{\frac{N-2}{4}}\left({1+|x|^{2}}\right)^{-\frac{N-2}{2}}.
\]
Hypothesis (\ref{eq: GNN condition}) was removed by Caffarelli, Gidas
and Spruck in \cite{Caffarelli1989}.  Thus for positive $C^{2}$
solutions $u$ of equation (\ref{eq: mu is 0}) when $p=2$, there
exists $\la>0$ and $x_{0}\in\R^{N}$ such that $u=u_{0}^{\la,x_{0}}$.
Hence we have that
\[
\lim_{|x|\to0}|\na u_{0}^{\la,x_{0}}(x)||x|=0,
\]
and that
\[
\lim_{|x|\to\infty}|\na u_{0}^{\la,x_{0}}(x)||x|^{N-1}=C\la^{-\frac{N-2}{2}},
\]
for some constant $C=C(N)>0$.

In the general case when $p\in(1,N)$, we can follow the argument
of \cite[Theorem 3.13]{B} to find that weak positive radial solutions
in ${\cal D}^{1,p}(\R^{N})$ to equation (\ref{eq: mu is 0}) are
of the form $u(x)=u_{0}^{\la,x_{0}}(x)=\lambda^{\frac{N-p}{p}}u_{0}(\lambda(x-x_{0}))$
for some $\lambda>0$ and some $x_{0}\in\mathbb{R}^{N}$,  where $u_{0}\in{\cal D}^{1,p}(\R^{N})$
is a particular weak positive radial solution satisfying
\begin{eqnarray*}
\lim_{|x|\rightarrow0}|\nabla u_{0}(x)||x|=0 & \text{and} & \lim_{|x|\rightarrow\infty}|\nabla u_{0}(x)||x|^{\frac{N-1}{p-1}}=C
\end{eqnarray*}
for some positive constant $C=C(N,p)>0$. Thus for weak positive radial
solution $u=u_{0}^{\la,x_{0}}\in{\cal D}^{1,p}(\R^{N})$, we have
that
\[
\lim_{|x|\rightarrow0}|\nabla u_{0}^{\la,x_{0}}(x)||x|=0,
\]
and that
\[
\lim_{|x|\to\infty}|\na u_{0}^{\la,x_{0}}(x)||x|^{\frac{N-1}{p-1}}=C\la^{-\frac{N-p}{p}},
\]
for some positive constant $C=C(N,p)>0$.

In the case when $\mu\in(0,\bar{\mu})$, a prototype of equation (\ref{eq: Objects})
when $Q\equiv$1 is
\begin{eqnarray}
-\Delta_{p}u-\frac{\mu}{|x|^{p}}|u|^{p-2}u=|u|^{p^{*}-2}u, &  & \text{in }\R^{N}.\label{eq: object of p not equal to 2}
\end{eqnarray}
When $p=2$, by Chou and Chu \cite[Theorem B]{ChouKS1993}, every
positive solution $u\in C^{2}(\R^{N}\backslash\{0\})$ must be radially
symmetric with respect to the origin, provided that $u$ satisfies
\begin{equation}
|x|^{\sqrt{\overline{\mu}}-\sqrt{\overline{\mu}-\mu}}u(x)\in L_{loc}^{\infty}(\R^{N}).\label{eq: CC condition}
\end{equation}
Catrina and Wang \cite{C} and Terracini \cite{Terracini} proved
that every positive radial solution of equation (\ref{eq: object of p not equal to 2})
must be of the form $u(x)=u_{0}^{\lambda}(x)=\lambda^{\frac{N-2}{2}}u_{0}(\lambda x)$
for some $\ensuremath{\lambda>0}$ , where $\ensuremath{u_{0}}$ is
given by
\[
u_{0}(x)=\left({4N(\bar{\mu}-\mu)}/{(N-2)}\right)^{\frac{N-2}{4}}\left(|x|^{\frac{\sqrt{\overline{\mu}}-\sqrt{\overline{\mu}-\mu}}{\sqrt{\bar{\mu}}}}+|x|^{\frac{\sqrt{\overline{\mu}}+\sqrt{\overline{\mu}-\mu}}{\sqrt{\bar{\mu}}}}\right)^{-\frac{N-2}{2}}.
\]
Thus for positive solution $u$ in $C^{2}(\R^{N}\backslash\{0\})$
satisfying (\ref{eq: CC condition}), there is a constant $\la>0$
such that $u(x)=u_{0}^{\lambda}(x)=\lambda^{\frac{N-2}{2}}u_{0}(\lambda x)$.
From which, we have that
\[
\lim_{|x|\to0}|\na u_{0}^{\la}(x)||x|^{\sqrt{\bar{\mu}}-\sqrt{\bar{\mu}-\mu}+1}=C_{1}\la^{\sqrt{\bar{\mu}-\mu}},
\]
 and that
\[
\lim_{|x|\to\infty}|\na u_{0}^{\la}(x)||x|^{\sqrt{\bar{\mu}}+\sqrt{\bar{\mu}-\mu}+1}=C_{2}\la^{-\sqrt{\bar{\mu}-\mu}},
\]
for some constants $C_{1},C_{2}>0$ depending only on $N$ and $\mu$.
We remark that, by (\ref{eq: origin growth}) of Theorem \ref{thm: Cao-Xiang-Yan},
every weak solution $u\in{\cal D}^{1,2}(\R^{N})$ of equation (\ref{eq: object of p not equal to 2})
satisfies hypothesis (\ref{eq: CC condition}).

In the general case when $p\in(1,N)$, Boumediene, Veronica and Peral
\cite[Theorem 3.13]{B} proved that   all  weak positive radial solutions
in ${\cal D}^{1,p}(\R^{N})$ of equation (\ref{eq: object of p not equal to 2})
are of the form $\ensuremath{u(x)=u_{0}^{\lambda}(x)=\lambda^{\frac{N-p}{p}}u_{0}(\lambda x)}$
for some $\ensuremath{\lambda>0}$, where $\ensuremath{u_{0}}$ is
a particular weak positive radial solution in ${\cal D}^{1,p}(\R^{N})$
satisfying     \begin{eqnarray} \lim_{|x|\rightarrow0}|\nabla u_0(x)||x|^{\gamma_{1}+1}=C_{1}  &\text{and}&\lim_{|x|\rightarrow \infty}|\nabla u_0(x)||x|^{\gamma_{2}+1}=C_{2},\label{est: optimal result} \end{eqnarray} for
some  constants $C_{1},C_{2}>0$. Thus for any weak positive radial
solution $u=u_{0}^{\la}$ of equation (\ref{eq: object of p not equal to 2}),
we have that
\begin{equation}
\lim_{|x|\rightarrow0}|\nabla u_{0}^{\la}(x)||x|^{\gamma_{1}+1}=C_{1}\la^{\frac{N-p}{p}-\ga_{1}},\label{est: optimal result 2}
\end{equation}
and that
\begin{equation}
\lim_{|x|\rightarrow\infty}|\nabla u_{0}^{\la}(x)||x|^{\gamma_{2}+1}=C_{2}\la^{\frac{N-p}{p}-\ga_{2}},\label{est: optimal result 3}
\end{equation}
with the constants $C_{1},C_{2}>0$ given by \eqref{est: optimal result}.

In this note, we give the asymptotic estimates for the gradient of
weak solutions to equation (\ref{eq: Objects}) both at the origin
and at the infinity.
\begin{thm}
\label{thm: main result}Let $Q\in L^{\infty}(\mathbb{R}^{N})$ and
$u\in\mathcal{D}^{1,p}(\mathbb{R}^{N})$ be a weak solution of equation
(\ref{eq: Objects}). Then there exists a positive constant $C$ depending
on $N,p,\mu,||Q||_{\infty}$ and $u$, such that
\begin{equation}
|\na u(x)|\leq{C}{|x|^{-\gamma_{1}-1}}\quad for\;|x|<R_{0},\label{eq: origin growth 1}
\end{equation}
 and that
\begin{equation}
|\na u(x)|\leq{C}{|x|^{-\gamma_{2}-1}}\quad for\;|x|>R_{1},\label{eq: infinity decay-1}
\end{equation}
where $0<R_{0}<1<R_{1}$ depend on $N,p,\mu,||Q||_{\infty}$ and $u$.
\end{thm}
Again, in the above theorem the positive constants $C,R_{0},R_{1}$
depend on the solution $u$. Indeed, this is the case, since equation
\eqref{eq: Objects} when $Q\equiv1$  is invariant under the scaling
$v(x)=\lambda^{\frac{N-p}{p}}u(\lambda x)$, $\lambda>0$. In above
theorems and in the following, if we say a constant depends on the
solution $u$, it means that the constant depends on $||u||_{p^{*},\R^{N}}$,
the $L^{p^{*}}$-norm of $u$, and also on the modulus of continuity
of the function $h(r)=||u||_{p^{*},B_{r}(0)}+||u||_{p^{*},\R^{N}\backslash B_{1/r}(0)}$
at $r=0$. Precisely, we can choose a constant $\ep>0$ depending
on $N,p,\mu$ and $||Q||_{\infty}$. Since $h(r)\to0$ as $r\to0$,
there exists $r_{0}>0$ such that
\[
||u||_{p^{*},B_{r_{0}}(0)}+||u||_{p^{*},\R^{N}\backslash B_{1/r_{0}}(0)}<\ep.
\]
Then the constants $C,R_{0},R_{1}$ in Theorem \ref{thm: Cao-Xiang-Yan}
and Theorem \ref{thm: main result} depend also on $r_{0}$. The reader
is referred to find more details on this dependence in \cite{CaoXiangYan2014}.

Estimates (\ref{est: optimal result 2}) and (\ref{est: optimal result 3})
imply that the exponents $\ga_{1}+1$ and $\ga_{2}+1$ in the estimates
(\ref{eq: origin growth 1}) and (\ref{eq: infinity decay-1}) respectively
are optimal.

The idea to prove Theorem \ref{thm: main result} is as follows. Let
$u$ be a weak solution to equation (\ref{eq: Objects}) and set
\begin{eqnarray*}
f(x)=\mu|x|^{-p}|u|^{p-2}u+Q(x)|u|^{p^{*}-2}u, &  & x\in\R^{N}.
\end{eqnarray*}
Then $u$ is a weak solution to equation
\begin{eqnarray}
-\De_{p}u=f &  & \text{in }\R^{N}\backslash\{0\}.\label{eq: p-Laplacian equation 1}
\end{eqnarray}
For any ball $B_{|x|/2}(x)$ centered at $x$ with radius $|x|/2$,
$x\ne0$, gradient estimate of the $p$-Laplacian equation (\ref{eq: p-Laplacian equation 1})
gives us
\begin{equation}
\sup_{B_{|x|/8}(x)}|\na u|\le C\left(\fint_{B_{|x|/4}(x)}|\na u|^{p}\right)^{\frac{1}{p}}+C|x|^{\frac{1}{p-1}}||f||_{\infty,B_{|x|/4}(x)}^{\frac{1}{p-1}}.\label{eq: apriori gradient estimate 1}
\end{equation}
For the terms on the right hand side of (\ref{eq: apriori gradient estimate 1}),
Theorem \ref{thm: Cao-Xiang-Yan} gives estimates on the second term
at the origin and at the infinity. The estimate on the first term
follows from Caccioppoli inequality, see Lemma \ref{lem: Cacciopolli estimate}
in Section 2. So we obtain the estimates in Theorem \ref{thm: main result}
from (\ref{eq: apriori gradient estimate 1}).

The note is organized as follows. In Section 2, we prove Theorem \ref{thm: main result}.
In Section 3 we prove the gradient estimate of $p$-Laplacian equation.

Our notations are standard. $B_{R}(x)$ is the open ball in $\R^{N}$
centered at $x$ with radius $R>0$. We write
\[
\fint_{B_{R}(x)}u=\frac{1}{|B_{R}(x)|}\int_{B_{R}(x)}u,
\]
where $|B_{R}(x)|$ is the $n$-dimensional Lebesgue measure of $B_{R}(x)$.
Let $\Om$ be an arbitrary domain in $\R^{N}$. We denote by $C_{0}^{\infty}(\Om)$
the space of smooth functions with compact support in $\Om$. For
any $1\le q\le\infty$, $L^{q}(\Om)$ is the Banach space of Lebesgue
measurable functions $u$ such that the norm
\[
||u||_{q,\Om}=\begin{cases}
\left(\int_{\Om}|u|^{q}\right)^{\frac{1}{q}} & \text{if }1\le q<\infty\\
\esssup_{\Om}|f| & \text{if }q=\infty
\end{cases}
\]
is finite. The local space $L_{\text{loc}}^{q}(\Om)$ consists of
functions belonging to $L^{q}(\Om^{\prime})$ for all $\Om^{\prime}\subset\subset\Om$.
A function $u$ belongs to the Sobolev space $W^{1,q}(\Om)$ if $u\in L^{q}(\Om)$
and its first order weak partial derivatives also belong to $L^{q}(\Om)$.
We endow $W^{1,q}(\Om)$ with the norm
\[
||u||_{1,q,\Om}=||u||_{q,\Om}+||\na u||_{q,\Om}.
\]
The local space $W_{\text{loc}}^{1,q}(\Om)$ consists of functions
belonging to $W^{1,q}(\Om^{\prime})$ for all open $\Om^{\prime}\subset\subset\Om$.
We recall that $W_{0}^{1,q}(\Om)$ is the completion of $C_{0}^{\infty}(\Om)$
in the norm $||\cdot||_{1,q,\Om}$. For the properties of the Sobolev
functions, we refer to the monograph \cite{Ziemer}.

\section{Proof of main result}

This section is devoted to the proof of Theorem \ref{thm: main result}.
We need the following results. The first result is the gradient estimate
for the $p$-Laplacian equation.
\begin{prop}
\label{lem:  apriori  gradient estimate} Let $\Om$ be a domain in
$\R^{N}$ and $f\in L_{\text{loc}}^{\infty}(\Om)$. Let $u\in W_{\text{loc }}^{1,p}(\Om)$
be a weak solution to equation
\begin{equation}
-\De_{p}u=f\label{eq: p-Laplacian equation}
\end{equation}
in $\Om$, that is,
\begin{eqnarray*}
\int_{\Om}|\na u|^{p-2}\na u\cdot\na\var=\int_{\Om}f\var, &  & \forall\,\var\in C_{0}^{\infty}(\Om).
\end{eqnarray*}
Then for any ball $B_{2R}(x_{0})\subset\Om$, there holds
\begin{equation}
\sup_{B_{R/2}(x_{0})}|\na u|\le C\left(\fint_{B_{R}(x_{0})}|\na u|^{p}\right)^{\frac{1}{p}}+CR^{\frac{1}{p-1}}||f||_{\infty,B_{R}(x_{0})}^{\frac{1}{p-1}},\label{eq: gradient estimate}
\end{equation}
where $C>0$ depends only on $N$ and $p$.
\end{prop}
Proposition \ref{lem:  apriori  gradient estimate} is well known.
In the case $f\equiv0$, Proposition \ref{lem:  apriori  gradient estimate}
has been proved by DiBenedetto \cite[Proposition 3]{DiBenedetto1983}.
We will follow the argument of DiBenedetto \cite{DiBenedetto1983}
to prove Proposition \ref{lem:  apriori  gradient estimate} in the
next section.

The second result is a consequence of Theorem \ref{thm: Cao-Xiang-Yan}.
\begin{lem}
\label{lem: Cacciopolli estimate} Let $Q\in L^{\infty}(\mathbb{R}^{N})$
and $u\in\mathcal{D}^{1,p}(\mathbb{R}^{N})$ be a weak solution to
equation (\ref{eq: Objects}). Let $R_{0},R_{1}$ be the constants
as in Theorem \ref{thm: Cao-Xiang-Yan}. Then there exists a positive
constant $C$ depending on $N,p,\mu,||Q||_{\infty}$ and the solution
$u$ such that
\begin{eqnarray}
\fint_{B_{|x|/4}(x)}|\na u|^{p}\le C|x|^{-p(\gamma_{1}+1)} &  & \text{for }\;0<|x|<R_{0}/2,\label{eq: origin growth of gradient 1}
\end{eqnarray}
 and that
\begin{eqnarray}
\fint_{B_{|x|/4}(x)}|\na u|^{p}\le C|x|^{-p(\gamma_{2}+1)} &  & for\;|x|>2R_{1}.\label{eq: infinity decay of gradient 1}
\end{eqnarray}
\end{lem}
\begin{proof}
Fix $x\in\R^{N}$ such that $0<|x|<R_{0}/2$. Let $B=B_{|x|/4}(x)$
and $2B=B_{|x|/2}(x)$. Let $\eta\in C_{0}^{\infty}(2B)$ be a cut-off
function such that $0\le\eta\le1$ in $2B$ and $\eta\equiv1$ on
$B$, $|\na\eta|\le8/|x|$. Substituting test function $\var=\eta^{p}u$
into equation (\ref{eq: Objects}), we obtain that
\[
\int_{2B}|\na u|^{p-2}\na u\cdot\na\var=\int_{2B}\left(\frac{\mu}{|y|^{p}}\eta^{p}|u|^{p}+Q(y)\eta^{p}|u|^{p^{*}}\right).
\]
We have that
\[
\int_{2B}|\na u|^{p-2}\na u\cdot\na\var\ge\frac{1}{2}\int_{B}|\na u|^{p}-C_{p}\int_{2B}|u|^{p}|\na\eta|^{p}
\]
for some constant $C_{p}>0$ depending only on $p$. Thus
\[
\int_{B}|\na u|^{p}\le C_{p}\int_{2B}\left(|u|^{p}|\na\eta|^{p}+\frac{\mu}{|y|^{p}}\eta^{p}|u|^{p}+Q(y)\eta^{p}|u|^{p^{*}}\right).
\]
Applying (\ref{eq: origin growth}) of Theorem \ref{thm: Cao-Xiang-Yan},
we obtain that
\begin{eqnarray*}
\int_{B}|\na u|^{p}\le C|x|^{-p-p\ga_{1}+N}, &  & \forall\,0<|x|<R_{0}/2,
\end{eqnarray*}
where $C>0$ depends on $N,p,\mu,||Q||_{\infty}$ and the solution
$u$. This proves (\ref{eq: origin growth of gradient 1}). We can
prove (\ref{eq: infinity decay of gradient 1}) similarly. We finish
the proof of Lemma \ref{lem: Cacciopolli estimate}.
\end{proof}
Now we prove Theorem \ref{thm: main result}.

\begin{proof}[Proof of Theorem \ref{thm: main result}] Let $u\in{\cal D}^{1,p}(\R^{N})$
be a solution to equation (\ref{eq: Objects}) with $Q\in L^{\infty}(\R^{N})$.
We only prove (\ref{eq: origin growth 1}). We can prove (\ref{eq: infinity decay-1})
similarly. Let $R_{0}\in(0,1)$ be the constant as in Theorem \ref{thm: Cao-Xiang-Yan}.
Set
\begin{eqnarray*}
f(x)\equiv\frac{\mu}{|x|^{p}}|u|^{p-2}u+Q(x)|u|^{p^{*}-2}u, &  & x\in B_{R_{0}}(0)\backslash\{0\}.
\end{eqnarray*}
By (\ref{eq: origin growth}) of Theorem \ref{thm: Cao-Xiang-Yan}
and the fact that $\ga_{1}<(N-p)/p$, we obtain that
\begin{eqnarray}
|f(x)|\le C|x|^{-p-(p-1)\ga_{1}} &  & \forall\,0<|x|<R_{0}.\label{eq: estimate on f}
\end{eqnarray}
Thus $f\in L_{\text{loc}}^{\infty}(B_{R_{0}}(0)\backslash\{0\})$.

Since $u$ is a weak solution to equation (\ref{eq: Objects}), $u$
is a weak solution to equation (\ref{eq: p-Laplacian equation}) in
$B_{R_{0}}(0)\backslash\{0\}$ with $f$ given above. For any $x\in B_{R_{0}}(0)\backslash\{0\}$,
we apply Proposition \ref{lem:  apriori  gradient estimate} on the
ball $B_{|x|/2}(x)$ to obtain that
\begin{equation}
\sup_{B_{|x|/8}(x)}|\na u|\le C\left(\fint_{B_{|x|/4}(x)}|\na u|^{p}\right)^{\frac{1}{p}}+C|x|^{\frac{1}{p-1}}||f||_{\infty,B_{|x|/4}(x)}^{\frac{1}{p-1}}.\label{eq: gradient estimates 2}
\end{equation}
Combining (\ref{eq: origin growth of gradient 1}), (\ref{eq: estimate on f})
and (\ref{eq: gradient estimates 2}) gives that
\begin{eqnarray*}
\sup_{B_{|x|/8}(x)}|\na u|\le C|x|^{-1-\ga_{1}} &  & \forall\,0<|x|<R_{0}/2,
\end{eqnarray*}
for some constant $C>0$ depending on $N,p,\mu,||Q||_{\infty}$ and
the solution $u$. This proves (\ref{eq: origin growth 1}). \end{proof}

\section{Gradient estimates for $p$-Laplacian equations}

This section is devoted to the proof of Proposition \ref{lem:  apriori  gradient estimate}.

Let $B_{2R}(x_{0})\subset\Om$ be an arbitrary ball. In the following
we write $B_{r}=B_{r}(x_{0})$ for all $r>0$. Let $\ep>0$. Following
\cite{DiBenedetto1983}, we consider the equation
\begin{eqnarray}
\begin{cases}
-\divergence\left((\ep+|\na u_{\ep}|^{2})^{\frac{p-2}{2}}\na u_{\ep}\right)=f & \text{in }B_{2R},\\
u_{\ep}=u & \text{on }\pa B_{2R}.
\end{cases}\label{eq: approximation equation}
\end{eqnarray}
Then (\ref{eq: approximation equation}) admits a unique solution
$u_{\ep}\in W^{1,p}(B_{2R})$ such that
\[
u_{\ep}\in C^{2}(B_{2R}),
\]
and up to a subsequence
\[
u_{\ep}\to u\text{ and }\na u_{\ep}\to\na u\text{ uniformly in }B_{R}
\]
as $\ep\to0$.

To prove Proposition \ref{lem:  apriori  gradient estimate}, we will
prove the following estimate for $u_{\ep}$:
\begin{equation}
\sup_{B_{R/2}}|\na u_{\ep}|\le C\left(\fint_{B_{R}}\left(\ep+|\na u_{\ep}|^{2}\right)^{\frac{p}{2}}\right)^{\frac{1}{p}}+CR^{\frac{1}{p-1}}||f||_{\infty,B_{R}}^{\frac{1}{p-1}},\label{eq: gradient estimate for u_epsilon}
\end{equation}
for a constant $C>0$ depending only on $N$ and $p$ and independent
of $\ep$ and $R$. Then by taking $\ep\to0$ in (\ref{eq: gradient estimate for u_epsilon}),
we obtain (\ref{eq: gradient estimate}) and then Proposition \ref{lem:  apriori  gradient estimate}
is proved.

We divide the proof of (\ref{eq: gradient estimate for u_epsilon})
into several lemmas. For simplicity, we write $v=u_{\ep}$ and $w=\ep+|\na v|^{2}$.
We shall always assume that $N\ge3$. We can prove (\ref{eq: gradient estimate for u_epsilon})
similarly when $N=2$. First we derive the following Caccioppoli type
inequality.
\begin{lem}
\label{lem: Caccioppoli inequality}For any $\al\ge\max(p-2,0)$ and
any $\eta\in C_{0}^{\infty}(B_{R})$, we have
\begin{equation}
\int_{B_{R}}w^{\frac{\al+p-4}{2}}|\na w|^{2}\eta^{2}\le C\int_{B_{R}}w^{\frac{\al+p}{2}}|\na\eta|^{2}+C\int_{B_{R}}|f|^{2}w^{\frac{\al+2-p}{2}}\eta^{2},\label{eq: Cacciopolli inequality}
\end{equation}
for some $C=C(N,p)>0$. \end{lem}
\begin{proof}
For simplicity, we write $\pa_{i}=\pa_{x_{i}},\pa_{ij}=\pa_{x_{i}x_{j}}^{2}$
$(i,j=1,\cdots,N)$. Differentiating equation (\ref{eq: approximation equation})
with respect to $x_{k}$ $(k=1,\cdots,N)$ gives
\begin{eqnarray*}
-\pa_{i}({\cal A}^{ij}(\na v)\pa_{jk}v)=\pa_{k}f &  & \text{in }B_{2R},
\end{eqnarray*}
where
\begin{eqnarray*}
{\cal A}^{ij}(\na v)=\Big(\ep+|\na v|^{2}\Big)^{\frac{p-2}{2}}\de_{ij}+(p-2)(\ep+|\na v|^{2})^{\frac{p-4}{2}}\pa_{i}v\pa_{j}v, &  & i,j=1,\cdots,N,
\end{eqnarray*}
$\de_{ij}=1$ if $i=j$ and $\de_{ij}=0$ if $i\ne j$. The above
equation is understood in the sense that, for all $\var\in C_{0}^{\infty}(B_{2R})$,
(the summation notation is used throughout)
\begin{equation}
\int_{B_{2R}}{\cal A}^{ij}(\na v)\pa_{jk}v\pa_{i}\var=-\int_{B_{2R}}f\pa_{k}\var.\label{eq: equation of gradient u epsilon}
\end{equation}
 It is easy to prove that (\ref{eq: equation of gradient u epsilon})
holds also for all $\var\in W_{0}^{1,q}(B_{2R})$ for any $q\ge1$.
Set
\[
\var=w^{\frac{\al}{2}}\pa_{k}v\eta^{2},
\]
where $\eta\in C_{0}^{\infty}(B_{R})$ and $\al\ge\max(p-2,0)$. Then
\begin{eqnarray*}
\pa_{i}\var=w^{\frac{\al}{2}}\pa_{ik}v\eta^{2}+\frac{\al}{2}w^{\frac{\al}{2}-1}\pa_{i}w\pa_{k}v\eta^{2}+2w^{\frac{\al}{2}}\pa_{k}v\eta\pa_{i}\eta, &  & i=1,\cdots,N.
\end{eqnarray*}
Substituting $\var$ into equation (\ref{eq: equation of gradient u epsilon}),
and summing up all $k=1,\cdots,N$, we obtain that
\begin{equation}
\begin{aligned}\sum_{k=1}^{N}\int_{B_{R}}{\cal A}^{ij}(\na v)\pa_{jk}v\pa_{i}\var & \ge C_{1}\int_{B_{R}}w^{\frac{\al+p-2}{2}}|\na^{2}v|^{2}\eta^{2}+C_{1}(\al+p)\int_{B_{R}}w^{\frac{\al+p-4}{2}}|\na w|^{2}\eta^{2},\\
 & \quad-\, C_{2}\int_{B_{R}}w^{\frac{\al+p}{2}}|\na\eta|^{2},
\end{aligned}
\label{eq: big term}
\end{equation}
where $C_{1},C_{2}>0$ depend only on $p$, $|\na^{2}v|=\left(\sum_{i,j=1}^{N}(\pa_{ij}v)^{2}\right)^{1/2}$
, and that
\[
\sum_{k=1}^{N}\left|\int_{B_{R}}f\pa_{k}\var\right|\le\int_{B_{R}}|f|\left(w^{\frac{\al}{2}}|\na^{2}v|\eta^{2}+\al w^{\frac{\al-1}{2}}|\na w|\eta^{2}+2w^{\frac{\al+1}{2}}\eta|\na\eta|\right).
\]
Applying Young's inequality
\begin{eqnarray}
a^{\th}b^{1-\th}\le\th\de a+\frac{1-\theta}{\de^{\frac{\theta}{1-\th}}}b, &  & \forall\,\de,a,b>0,\forall\th\in[0,1)\label{eq: Young's inequality}
\end{eqnarray}
we obtain that
\begin{equation}
\begin{aligned}\sum_{k=1}^{N}\left|\int_{B_{R}}f\pa_{k}\var\right| & \le\frac{C_{1}}{2}\int_{B_{R}}w^{\frac{\al+p-2}{2}}|\na^{2}v|^{2}\eta^{2}+\frac{C_{1}(\al+p)}{2}\int_{B_{R}}w^{\frac{\al+p-4}{2}}|\na w|^{2}\eta^{2}\\
 & \quad+C\int_{B_{R}}w^{\frac{\al+p}{2}}|\na\eta|^{2}+C(\al+p)\int_{B_{R}}|f|^{2}w^{\frac{\al+2-p}{2}}\eta^{2}.
\end{aligned}
\label{eq: small term}
\end{equation}
Combining (\ref{eq: big term}) and (\ref{eq: small term}), we obtain
(\ref{eq: Cacciopolli inequality}) for some $C>0$ depends only on
$N$ and $p$. This finishes the proof of Lemma \ref{lem: Caccioppoli inequality}.
\end{proof}
By the Sobolev inequality we obtain the following reverse inequality.
\begin{lem}
\label{lem: Reverse Holder inequality} For any $\al\ge\max(p-2,0)$
and $\eta\in C_{0}^{\infty}(B_{R})$, we have
\begin{equation}
\left(\int_{B_{R}}\left(\eta^{2}w^{\frac{\al+p}{2}}\right)^{\chi}\right)^{1/\chi}\le C_{N,p}(\al+p)^{2}\left(\int_{B_{R}}w^{\frac{\al+p}{2}}|\na\eta|^{2}+\int_{B_{R}}|f|^{2}w^{\frac{\al+2-p}{2}}\eta^{2}\right),\label{eq: reverse Holder inequality-1}
\end{equation}
where $\chi=N/(N-2)$.\end{lem}
\begin{proof}
Let $h=\eta w^{\frac{\al+p}{4}}$. Then
\[
|\na h|^{2}\le2|\na\eta|^{2}w^{\frac{\al+p}{2}}+(\al+p)^{2}w^{\frac{\al+p-4}{2}}|\na w|^{2}\eta^{2}.
\]
By (\ref{eq: Cacciopolli inequality}) of Lemma \ref{lem: Caccioppoli inequality},
\begin{equation}
\int_{B_{R}}|\na h|^{2}\le C(\al+p)^{2}\int_{B_{R}}\left(w^{\frac{\al+p}{2}}|\na\eta|^{2}+|f|^{2}w^{\frac{\al+2-p}{2}}\eta^{2}\right),\label{eq: 3.2.1}
\end{equation}
where $C=C(N,p)>0$. Now we use Sobolev inequality to obtain
\begin{equation}
\left(\int_{B_{R}}h^{2\chi}\right)^{\frac{1}{\chi}}\le C_{N}\int_{B_{R}}|\na h|^{2},\label{eq: Sobolev}
\end{equation}
where $\chi=N/(N-2)$. Combining (\ref{eq: 3.2.1}) and (\ref{eq: Sobolev})
yields (\ref{eq: reverse Holder inequality-1}). We finish the proof
of Lemma \ref{lem: Reverse Holder inequality}.
\end{proof}
In the following, we write
\begin{eqnarray}
\bar{w}=w^{1/2} & \text{ and } & F(r)=\left(r||f||_{\infty,B_{r}}\right)^{1/(p-1)}.\label{eq: definition of bar(w)}
\end{eqnarray}
As a consequence of Lemma \ref{lem: Reverse Holder inequality}, we
have
\begin{cor}
\label{cor: Reverse inequality -2} Let $0<r\le R$ and $\al\ge\max(p-2,0)$.
Then for any $0<r_{1}<r_{2}\le r$, we have
\begin{equation}
\left(\int_{B_{r_{1}}}\left(\bar{w}^{(\al+p)\chi}+F(r)^{(\al+p)\chi}\right)\right)^{\frac{1}{\chi}}\le\frac{C_{N,p}(\al+p)^{2}}{(r_{2}-r_{1})^{2}}\int_{B_{r_{2}}}\left(\bar{w}^{\al+p}+F(r)^{\al+p}\right),\label{eq: iteration formular 1}
\end{equation}
where $\chi=N/(N-2)$. \end{cor}
\begin{proof}
Let $\eta\in C_{0}^{\infty}(B_{r_{2}})$ be a cut-off function such
that $0\le\eta\le1$ in $B_{r_{2}}$, $\eta\equiv1$ on $B_{r_{1}}$
and $|\na\eta|\le2/(r_{2}-r_{1})$. Substituting $\eta$ into (\ref{eq: reverse Holder inequality-1})
we obtain that
\[
\left(\int_{B_{r_{1}}}w^{\frac{(\al+p)\chi}{2}}\right)^{1/\chi}\le C_{N,p}(\al+p)^{2}\left(\frac{1}{(r_{2}-r_{1})^{2}}\int_{B_{r_{2}}}w^{\frac{\al+p}{2}}+||f||_{\infty,B_{r}}^{2}\int_{B_{r_{2}}}w^{\frac{\al+2-p}{2}}\right).
\]
Thus we have that
\[
\left(\int_{B_{r_{1}}}w^{\frac{(\al+p)\chi}{2}}\right)^{1/\chi}\le\frac{C_{N,p}(\al+p)^{2}}{(r_{2}-r_{1})^{2}}\left(\int_{B_{r_{2}}}w^{\frac{\al+p}{2}}+\int_{B_{r_{2}}}(r||f||_{\infty,B_{r}})^{2}w^{\frac{\al+2-p}{2}}\right).
\]
Since $\al\ge\max(p-2,0)$ and $p>1$, $\al+p>\al+2-p\ge0$. Young's
inequality (\ref{eq: Young's inequality}) gives
\[
(r||f||_{\infty,B_{r}})^{2}w^{\frac{\al+2-p}{2}}\le w^{\frac{\al+p}{2}}+\left(r||f||_{\infty,B_{r}}\right)^{\frac{\al+p}{p-1}}.
\]
Recall that $\bar{w}$ and $F(r)$ are defined by (\ref{eq: definition of bar(w)}).
Thus we have
\begin{eqnarray*}
\left(\int_{B_{r_{1}}}w^{\frac{(\al+p)\chi}{2}}\right)^{1/\chi} & \le & \frac{C_{N,p}(\al+p)^{2}}{(r_{2}-r_{1})^{2}}\left(\int_{B_{r_{2}}}w^{\frac{\al+p}{2}}+\int_{B_{r_{2}}}\left(r||f||_{\infty,B_{r}}\right)^{\frac{\al+p}{p-1}}\right)\\
 & = & \frac{C_{N,p}(\al+p)^{2}}{(r_{2}-r_{1})^{2}}\int_{B_{r_{2}}}\left(\bar{w}^{\al+p}+F(r)^{\al+p}\right).
\end{eqnarray*}
Since $r_{1}<r_{2}\le r$, we have $|B_{r_{1}}|^{\frac{1}{\chi}}\le C_{N}|B_{r_{2}}|/(r_{2}-r_{1})^{2}$.
Therefore
\[
\begin{aligned}\left(\int_{B_{r_{1}}}\left(\bar{w}^{(\al+p)\chi}+F(r)^{(\al+p)\chi}\right)\right)^{\frac{1}{\chi}} & \le2^{1-\frac{1}{\chi}}\left(\int_{B_{r_{1}}}\left(\bar{w}^{(\al+p)\chi}\right)\right)^{\frac{1}{\chi}}+2^{1-\frac{1}{\chi}}F(r)^{\al+p}|B_{r_{1}}|^{\frac{1}{\chi}}\\
 & \le\frac{C_{N,p}(\al+p)^{2}}{(r_{2}-r_{1})^{2}}\int_{B_{r_{2}}}\left(\bar{w}^{\al+p}+F(r)^{\al+p}\right)+\frac{C_{N}F(r)^{\al+p}|B_{r_{2}}|}{(r_{2}-r_{1})^{2}}.
\end{aligned}
\]
Now it is easy to obtain that
\[
\left(\int_{B_{r_{1}}}\left(\bar{w}^{(\al+p)\chi}+F(r)^{(\al+p)\chi}\right)\right)^{\frac{1}{\chi}}\le\frac{C_{N,p}(\al+p)^{2}}{(r_{2}-r_{1})^{2}}\left(\int_{B_{r_{2}}}\left(\bar{w}^{\al+p}+F(r)^{\al+p}\right)\right),
\]
which gives (\ref{eq: iteration formular 1}). We finish the proof
of Corollary \ref{cor: Reverse inequality -2}.
\end{proof}
Now we prove Proposition \ref{lem:  apriori  gradient estimate}.

\begin{proof}[Proof of Proposition \ref{lem:  apriori  gradient estimate}]
We prove estimate (\ref{eq: gradient estimate for u_epsilon}). Let
$\si\in(0,1)$ and $0<r\le R$. Let $r_{i}=\si r+\frac{(1-\si)r}{2^{i}},$
$i=0,1,\cdots$.

\emph{Case 1$:\ensuremath{1<p\le2}$. }In this case, define
\begin{eqnarray*}
 & \al_{i}=p\chi^{i}-p, & i=0,1,\cdots.
\end{eqnarray*}
 Applying (\ref{eq: iteration formular 1}) with $r_{1}=r_{i+1}$,
$r_{2}=r_{i}$ and $\al=\al_{i}$, we obtain that
\begin{eqnarray}
M_{i+1}\le\frac{C^{\frac{1}{p\chi^{i}}}(4\chi^{2})^{\frac{i}{p\chi^{i}}}}{\left((1-\si)r\right)^{\frac{2}{p\chi^{i}}}}M_{i}, &  & i=0,1,\cdots,\label{eq: iteration}
\end{eqnarray}
where
\[
M_{i}=\left(\int_{B_{r_{i}}}\left(\bar{w}^{p\chi^{i}}+F(r)^{p\chi^{i}}\right)\right)^{\frac{1}{p\chi^{i}}},
\]
and $\bar{w},F(r)$ are defined by (\ref{eq: definition of bar(w)}).
An iteration of (\ref{eq: iteration}) gives us
\[
M_{i+1}\le\frac{C_{N,p}}{(1-\si)^{N/p}r^{N/p}}\left(\int_{B_{r}}\left(\bar{w}^{p}+F(r)^{p}\right)\right)^{\frac{1}{p}}.
\]
Finally, letting $i\to\infty$, we obtain that
\[
\sup_{B_{\si r}}\left(\bar{w}+F(r)\right)\le\frac{C_{N,p}}{(1-\si)^{N/p}}\left(\fint_{B_{r}}\left(\bar{w}^{p}+F(r)^{p}\right)\right)^{\frac{1}{p}}\le\frac{C_{N,p}}{(1-\si)^{N/p}}\left(\left(\fint_{B_{r}}\bar{w}^{p}\right)^{\frac{1}{p}}+F(r)\right).
\]
In particular, choosing $\si=1/2$ and $r=R$, we obtain (\ref{eq: gradient estimate for u_epsilon})
for $1<p\le2$.

\emph{Case 2: $p>2$. }In this case, define
\begin{eqnarray*}
 & \al_{i}=(2p-2)\chi^{i}-p, & i=0,1,\cdots.
\end{eqnarray*}
Applying the same argument as above, we obtain that
\[
\sup_{B_{\si r}}\left(\bar{w}+F(r)\right)\le\frac{C_{N,p}}{(1-\si)^{N/(2p-2)}}\left(\fint_{B_{r}}\left(\bar{w}+F(r)\right)^{2p-2}\right)^{\frac{1}{2p-2}}.
\]
Since $\si<1$ and $F$ is nondecreasing, we obtain that
\begin{equation}
\sup_{B_{\si r}}\left(\bar{w}+F(\si r)\right)\le\frac{C_{N,p}}{(1-\si)^{N/(2p-2)}}\left(\fint_{B_{r}}\left(\bar{w}+F(r)\right)^{2p-2}\right)^{\frac{1}{2p-2}}.\label{eq:iteration formula 2}
\end{equation}
Let $\si_{i}=1-\frac{1-\si}{2^{i}},$ $i=0,1,\cdots.$ Applying (\ref{eq:iteration formula 2})
with $r=\si_{i+1}R$, $\si=\si_{i}/\si_{i+1}$, we get that
\begin{equation}
M_{i}\le\frac{C}{\left(1-\frac{{\displaystyle \si_{i}}}{{\displaystyle \si_{i+1}}}\right)^{N\beta/p}}\left(\fint_{B_{R}}\left(\bar{w}+F(R)\right)^{p}\right)^{\frac{\beta}{p}}M_{i+1}^{1-\be},\label{eq: iteration for p greater than 2-2}
\end{equation}
where $\be=p/(2p-2)$, and
\[
M_{i}=\sup_{B_{\si_{i}R}}\left(\bar{w}+F(\si_{i}R)\right).
\]
An iteration of (\ref{eq: iteration for p greater than 2-2}) gives
that
\[
\sup_{B_{\si R}}\left(\bar{w}+F(\si R)\right)=M_{0}\le\frac{C_{N,p}}{(1-\si)^{N/p}}\left(\fint_{B_{R}}\left(\bar{w}+F(R)\right)^{p}\right)^{\frac{1}{p}}.
\]
Choosing $\si=1/2$, and applying Minkowski's inequality, we obtain
(\ref{eq: gradient estimate for u_epsilon}) for $p>2$. Thus we complete
the proof of (\ref{eq: gradient estimate for u_epsilon}).

Now taking $\ep\to0$ in (\ref{eq: gradient estimate for u_epsilon}),
we obtain (\ref{eq: gradient estimate}). Proposition \ref{lem:  apriori  gradient estimate}
is proved. \end{proof}

\emph{Acknowledgement. }The author is financially supported by the
Academy of Finland, project 259224. He would like to thank Prof. Xiao
Zhong for his guidance in the preparation of this note.

\end{document}